\documentclass[letterpaper,english,cleveref, autoref,numberwithinsect,nolineno]{socg-lipics-v2021}

\pdfoutput=1
\hideLIPIcs

\usepackage[dvipsnames]{xcolor}
\usepackage[utf8]{inputenc}

\usepackage{amsmath}
\usepackage{multicol}
\usepackage{graphicx}
\usepackage[colorinlistoftodos]{todonotes}

\usepackage{amsfonts}
\usepackage{enumitem}
\usepackage{tikz}
\usepackage{tikz-cd}
\usetikzlibrary{decorations.markings,positioning}
\usepackage{amsthm}
\usepackage{amssymb}
\usepackage{cite}
\usepackage{float}
\usepackage{microtype}
\usepackage{mathtools}
\usepackage{hyperref}

\newcommand{\define}[1]{\textbf{\boldmath{#1}}}

\title{A Lattice-Theoretic Perspective on the Persistence Map}

\authorrunning{Brendan Mallery, Ad\'{e}lie Garin, Justin Curry}

\author{Brendan Mallery\footnote{Corresponding Author}}{Tufts University, Medford, Massachusetts, USA}{brendan.mallery@tufts.edu}{
}{}

\author{Ad\'{e}lie Garin}{\'{E}cole Polytechnique Fédérale de Lausanne (EPFL), Lausanne, Switzerland}{adelie.garin@epfl.ch}{https://orcid.org/0000-0002-3223-6320}{SNSF, CRSII5 177237}

\author{Justin Curry}{University at Albany, State University of New York, USA }{jmcurry@albany.edu}{https://orcid.org/
0000-0003-2504-8388}{Supported by NSF CCF-1850052 and NASA 80GRC020C0016}

\Copyright{Brendan Mallery, Ad\'{e}lie Garin, Justin Curry}




\funding{
 } 

\nolinenumbers

\begin{document}

\maketitle

\begin{abstract}
We provide a naturally isomorphic description of the persistence map from merge trees to barcodes \cite{trees_barcodesII} in terms of a monotone map from the partition lattice to the subset lattice. Our description is local, which offers the potential to speed up inverse computations, and brings classical tools in combinatorics to bear on an active area of research in topological data analysis (TDA).
\end{abstract}

\section{Background on the Inverse Problem}
Merge trees 
play a central role in topological data analysis (TDA).
One can apply persistent homology to a merge tree to obtain an ``adjacency free'' description of a merge tree in terms of its barcode, we call this association of a barcode to a merge tree the \emph{persistence map.}
Characterizing precisely how many merge trees map to the same barcode was studied in \cite{curry2017fiber,TRN,trees_barcodesII} and has yielded significant connections to geometric group theory, combinatorics, and statistics. 
Understanding the fiber of the persistence map is crucial for understanding how noise in data propagates to noise in persistent homology.

In \cite{TRN,trees_barcodesII} a combinatorial version of this inverse problem was considered; see Figure \ref{fig:comb_trees_barcodes}.
A \emph{combinatorial merge tree} is a binary, rooted, combinatorial tree with birth-ordered labels on the leaves $\{0,1,\ldots,n\}$ and death-ordered labels on the internal nodes. 
Every barcode with $n$ finite-length bars whose left (birth) endpoints are distinct and whose right (death) endpoints are distinct can be encoded by a \emph{combinatorial barcode} $B=\{(i,j)\}$ if the $i^{th}$ birth endpoint is matched with the $j^{th}$ death endpoint.
Equivalently, a combinatorial barcode is the graph of a permutation $\sigma$ of $\{1,\ldots,n\}$.


In this abstract, we characterize the persistence map from combinatorial merge trees to combinatorial barcodes in terms of monotone maps between two lattices: the subset lattice and the partition lattice. We show that a maximal chain in the subset and partition lattices corresponds to a combinatorial barcode and combinatorial merge tree respectively, and that one may incrementally construct solutions to the inverse problem using this correspondence.

\begin{figure}
    \centering
    \includegraphics[width= 0.9 \textwidth]{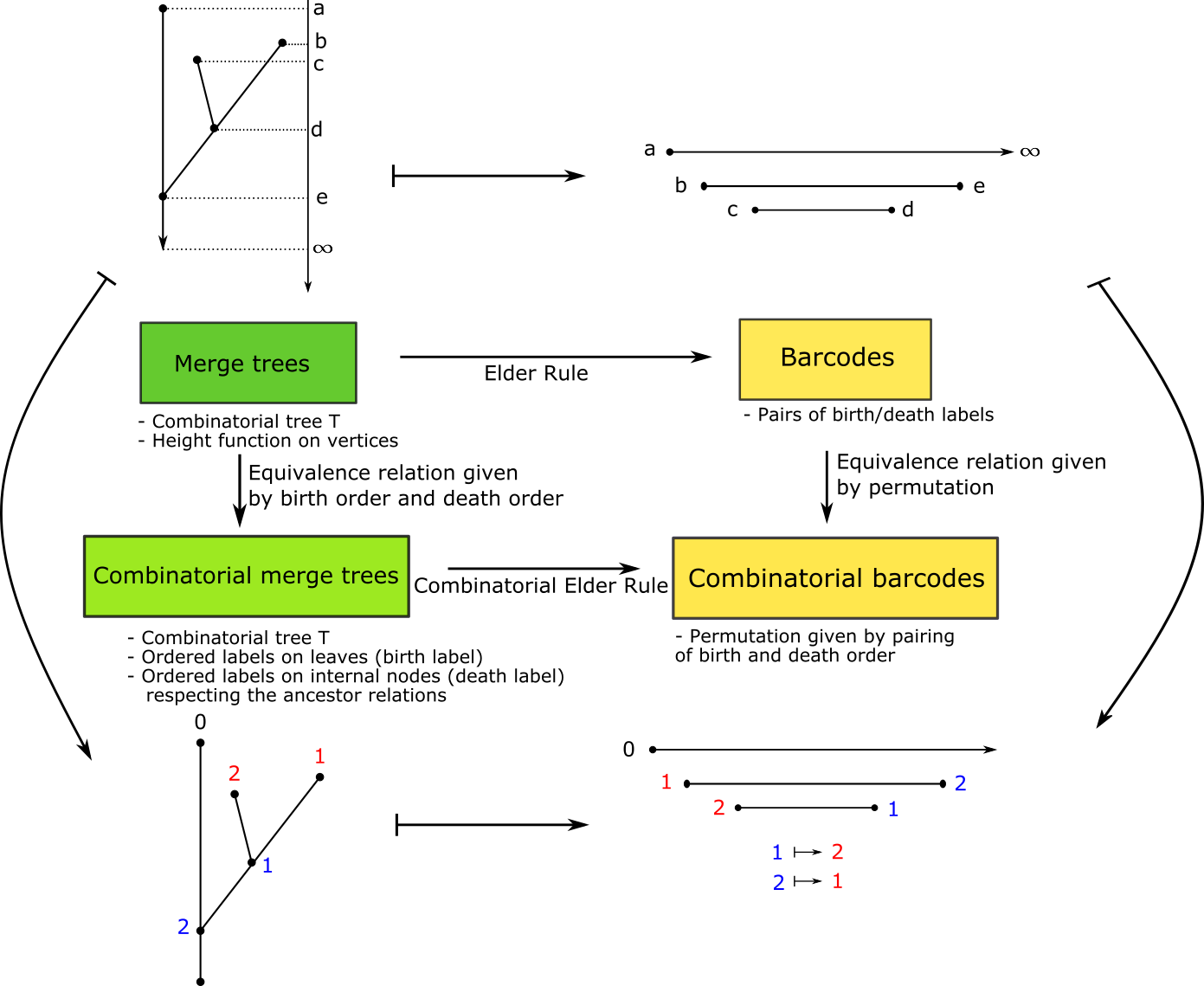}
    \caption{Figure from \cite{trees_barcodesII}, expressing the combinatorial inverse problem.}
    \label{fig:comb_trees_barcodes}
\end{figure}

\section{A Lattice Version of the Inverse Problem}

Let $(P, \preccurlyeq)$ be a poset.
Recall that a \emph{lattice} is a poset equipped with meets and joins.
A totally ordered subset $\mathcal{C} \subseteq P$ is called a \emph{chain}. 
An \emph{interval} is a subset $\mathcal{I}\subseteq P$ where if $p,q\in \mathcal{I}$ and $p\preccurlyeq r \preccurlyeq q$, then $r\in \mathcal{I}$.
A \emph{path} $\gamma$ is a chain that is also an interval. A path is \emph{based} at $x_0\in P$ if the lowest element in $\gamma$ is $x_0$. If $P$ has a unique lowest element $\hat{0}$ (e.g. a lattice), we write $\Tilde{P}$ as the poset of paths based at $\hat{0}$, which is a poset via containment of paths. There is a unique surjective map $\pi_P: \Tilde{P}\rightarrow P$ sending a path to its endpoint. Furthermore, if $f:P\rightarrow Q$ is a monotone map of posets, there is a unique map $\Tilde{f}:\Tilde{P}\rightarrow \Tilde{Q}$ such that $f\circ \pi_P=\pi_Q\circ \Tilde{f}$. We call $\Tilde{f}$ the \emph{lift} of $f$.


\begin{definition}[Subset Lattice]\label{ex_subset_lattice}
Let $[n]= \{1,\ldots, n\}$ and consider $P= \mathcal P ([n])$, the set of all subsets of $[n]$, including the empty set $\emptyset$, equipped with the partial order $ \subseteq$ of ``being a subset of''. This forms the \define{subset lattice} $\Pi_n$ of $[n]$, with $A\cap B$ and $A\cup B$ being the meet and join of $A,B\in\Pi_n$, respectively.
The poset of paths in $\Pi_n$ based at $\emptyset$ is $\tilde{\Pi}_n$.
\end{definition}

\begin{definition}[Partition Lattice]
A \emph{partition} of the set $\mathbf{n}:=\{0,1,\ldots,n\}$
is a collection of disjoint subsets $\mathcal{U}=\{U_1,\ldots, U_k\}$ of $\mathbf{n}$ whose union is $\mathbf{n}$.
A partition $\mathcal{U}$ \emph{refines} a partition $\mathcal{U}'$, written $\mathcal{U}\preceq \mathcal{U}'$, if every subset of $\mathcal{U}'$ is equal to a union of elements of $\mathcal{U}$. 
We denote the \define{lattice of partitions} of $\mathbf{n}$ by $\mathcal{P}_n$.
The poset of paths based at 
$\{\{0\},\ldots, \{n\}\}$
is $\tilde{\mathcal{P}}_n.$
\end{definition}




We can filter a combinatorial barcode $B$ with $n$ bars into sets $B_1\subset \cdots\subset B_n:=B$ where $B_k$ is the set of pairs $\{(i,j)\}_{j\leq k}$. We refer to $B_k$ as a \textit{partial (combinatorial) barcode}. The set of all partial barcodes with at most $n$ bars forms a poset by containment, which we denote by $\mathcal{PCB}_n$. Similarly, a \textit{partial (combinatorial) merge tree} is a filtration of a combinatorial merge tree $T$ with $n+1$ leaves by subgraphs $T_0\subset T_1\subset\cdots\subset T_n:= T$ where $T_k$ is the full subgraph supported on the set of leaf nodes and all internal nodes with label less than or equal to $k$. Partial merge trees also forms a poset by subgraph containment, denoted $\mathcal{PCT}_n$; see Figure \ref{lattice_summary}. The persistence map between combinatorial merge trees and barcodes extends to a map from $\mathcal{PCT}_n$ to $\mathcal{PCB}_n$, which we also call the persistence map.

\begin{theorem}\label{theorem_lattice}
The poset of partial merge trees $\mathcal{PCT}_n$ and barcodes $\mathcal{PCB}_n$ are isomorphic to $\tilde{\mathcal{P}}_n$ and $\Tilde{\Pi}_n$, respectively. Furthermore, there is a monotone map $H: \mathcal{P}_n \to \Pi_n$ whose lift $\Tilde{H}: \Tilde{\mathcal{P}}_n\rightarrow \Tilde{\Pi}_n$ is naturally isomorphic to the persistence map from $\mathcal{PCT}_n\rightarrow \mathcal{PCB}_n$.
\end{theorem}

\begin{proof} 
Every partial merge tree $T_0\subset \cdots\subset T_k$ defines a path $\mathcal{U}_0<\cdots <\mathcal{U}_k$, where $\mathcal{U}_i$ is the partition of the leaf node labels induced by connected components in the graph $T_i$; see Figure \ref{lattice_summary}. 
Every partial barcode $B_1\subset\cdots\subset B_k$ defines a path $\emptyset:=A_0\subset \cdots\subset A_k$, where $A_k$ is the set of birth labels whose deaths occur by time $k$. These specify the isomorphisms.

Define $H:\mathcal{P}_n\rightarrow \Pi_n$ as follows: Let $(U_1,U_2,...,U_k)$ be a partition of $\mathbf{n}$. For each $U_i$, let $U_i':=U_i\setminus \{\min\{x\in U_i\}\}$. Let $H((U_1,U_2,...,U_k))=\cup_{i\in [k]} U_i'\in \Pi_n$. This map is monotone, since if $(U_1,U_2,...,U_k)\leq (V_1,V_2,...,V_l)$, then the latter partition is obtained by collapsing parts of the first, which can only add elements to $H((U_1,U_2,...,U_k))$. It is easy to see that this map is also surjective. This lifts to a natural map $\Tilde{H}$, defined on paths. 

The maximal element (endpoint) of a path $\gamma\in \Tilde{\mathcal{P}}_n$ corresponds to a partition $(U_1,U_2,...,U_k)$ that indexes the leaf labels of the connected components of $T_k$, the $k^{th}$ stage in a partial merge tree.  
The Elder Rule \cite{curry2017fiber} of persistent homology maps each of the $U_i$ to $U_i'$ as $\min U_i$ encodes the oldest leaf node, which goes unpaired by the persistence algorithm.
The image is the union $B=\cup_{i\in [k]} B_i$ of leaf node labels that have been killed by stage $k$. 
The combinatorial barcode is encoded by the successive differences between $B_i$ and $B_{i+1}$.
\end{proof} 

\section{Future Work} 



Theorem \ref{theorem_lattice} is still in need of a full geometric description that accounts for actual positions and lengths of bars in a barcode and edges in a merge tree.
In \cite{barcode_coxeter} a novel coordinatization of barcode space was given based on the relation with the symmetric group.
However, a similar picture for merge tree space that uses the connection with the partition lattice is unknown.
Additionally, the lattice structure on these ``skeletonizations'' of barcode and merge tree space has not been fully explored.
As noted in \cite{gulen2022diagrams,mccleary2020edit,patel2018generalized}, M\"{o}bius inversion provides another way of summarizing topological changes in a filtration, which suggests that inverse problems, lattice theory, and M\"{o}bius inversion may occupy a rich intersection of ideas.



\bibliography{biblio}

\begin{thebibliography}{1}

\bibitem{barcode_coxeter}
Benjamin Brück and Adélie Garin.
\newblock Stratifying the space of barcodes using {C}oxeter complexes.
\newblock {\em ArXiv}, 2112.10571, 2021.

\bibitem{curry2017fiber}
Justin Curry.
\newblock The fiber of the persistence map for functions on the interval.
\newblock {\em Journal of Applied and Computational Topology}, 2(3):301--321,
  2018.

\bibitem{trees_barcodesII}
Justin Curry, Jordan DeSha, Adélie Garin, Kathryn Hess, Lida Kanari, and
  Brendan Mallery.
\newblock From trees to barcodes and back again {II}: Combinatorial and
  probabilistic aspects of a topological inverse problem.
\newblock {\em ArXiv}, 2107.11212v2, 2021.

\bibitem{gulen2022diagrams}
Aziz~Burak Gulen and Alexander McCleary.
\newblock Diagrams of persistence modules over finite posets.
\newblock {\em arXiv preprint arXiv:2201.06650}, 2022.

\bibitem{TRN}
Lida Kanari, Adélie Garin, and Kathryn Hess.
\newblock From trees to barcodes and back again: theoretical and statistical
  perspectives.
\newblock {\em Algorithms}, 13, 2020.
\newblock \href {http://dx.doi.org/10.3390/a13120335}
  {\path{doi:10.3390/a13120335}}.

\bibitem{mccleary2020edit}
Alexander McCleary and Amit Patel.
\newblock Edit distance and persistence diagrams over lattices.
\newblock {\em arXiv preprint arXiv:2010.07337}, 2020.

\bibitem{patel2018generalized}
Amit Patel.
\newblock Generalized persistence diagrams.
\newblock {\em Journal of Applied and Computational Topology}, 1(3):397--419,
  2018.

\end{thebibliography}

\begin{figure}
    \centering
    \includegraphics[width= \textwidth]{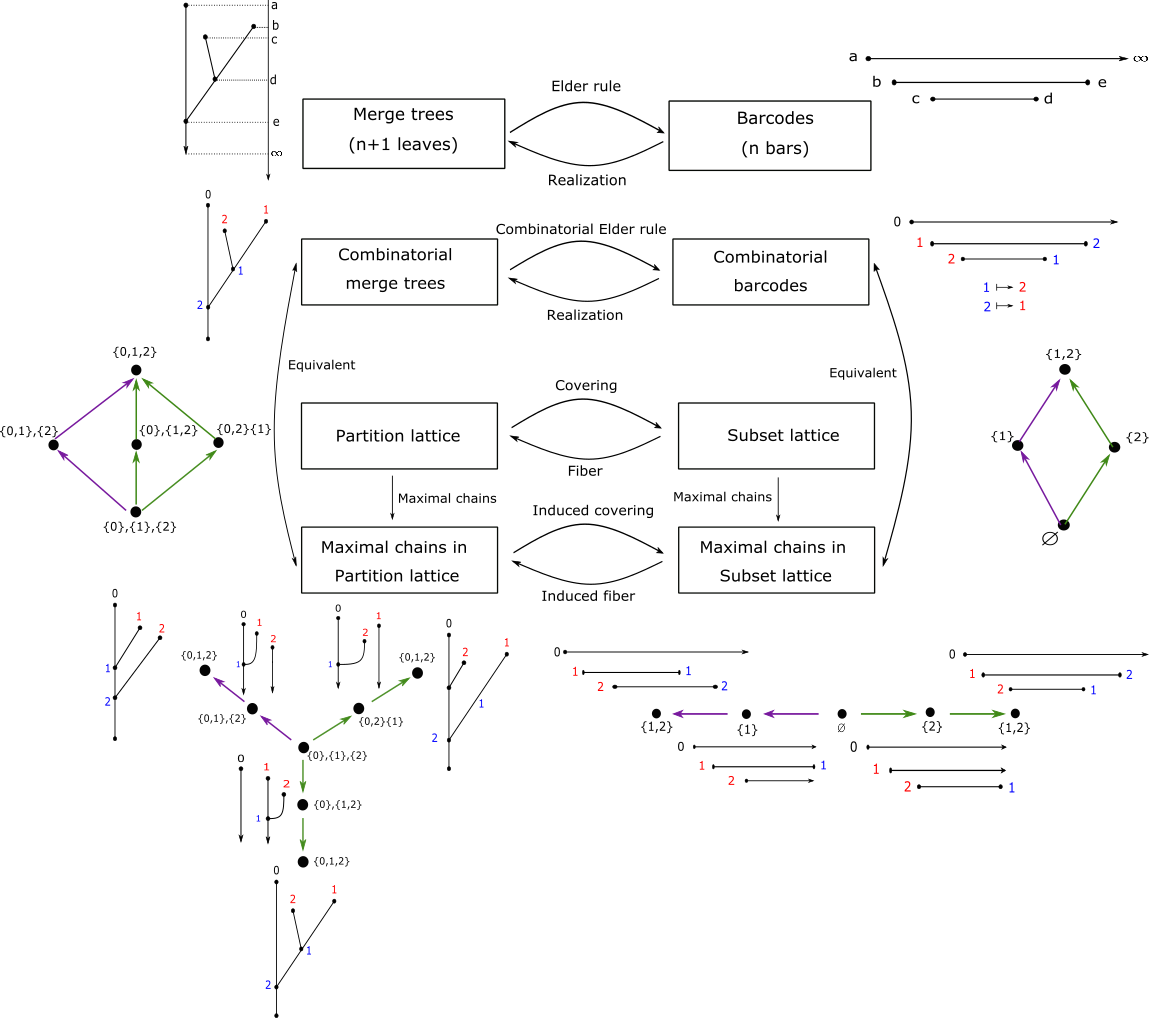}
    \caption{Illustration of Theorem \ref{theorem_lattice}.}
    \label{lattice_summary}
\end{figure}

\end{document}